\documentclass[11pt]{article}

\usepackage[margin=1in]{geometry} 
\usepackage{amsmath,amsthm,amssymb}
\usepackage{graphicx}
\usepackage{booktabs} 
\usepackage{enumerate}
\usepackage{float}
\usepackage[hyphens]{url}                   
\usepackage[colorlinks=true,
            linkcolor=blue,
            citecolor=blue,
            urlcolor=blue]{hyperref}
\usepackage{setspace}
\usepackage{sectsty}
\usepackage{titlesec}

\sectionfont{\centering\normalfont\normalsize\MakeUppercase}

\subsectionfont{\raggedright\normalsize\bfseries}

\titlespacing*{\section}{0pt}{3.5ex plus 1ex minus .2ex}{0.5ex plus .2ex}

\usepackage[authoryear,round]{natbib}
\let\cite\citet
            
\newtheorem{theorem}{\small\bf Theorem}

\theoremstyle{remark}

\theoremstyle{definition}
\newtheorem{definition}[theorem]{\small\bf Definition}

\newcommand{\ZZ}{\mathbb{Z}}
\newcommand{\EE}{\mathbb{E}}

\newcommand{\RR}{\mathbb{R}}

\newcommand{\COV}{\operatorname{Cov}}

\newcommand{\vect}{\operatorname{vec}}

\newcommand{\AR}[1]{\operatorname{AR}#1}
\newcommand{\VVAR}[1]{\operatorname{VAR}#1}

\newcommand{\MAR}[1]{\operatorname{MAR}#1}

 \DeclareMathOperator{\diag}{diag}

\date{}      

\newcommand{\keywords}[1]{%
  \bigskip\noindent\textbf{Keywords:} #1
}

\begin{document}

\onehalfspacing

\title{Estimation methods of Matrix‐valued AR model}
\author{%
  Kamil Kołodziejski, \\
  Email: kamil.kolodziejski@dokt.p.lodz.pl \\
  Institute of Mathematics, Łódź University of Technology
}
\maketitle

\begin{abstract}
  This article proposes novel estimation methods for the Matrix Autoregressive (MAR) model, specifically adaptations of the Yule-Walker equations and Burg’s method, addressing limitations in existing techniques.  The MAR model, by maintaining a matrix structure and requiring significantly fewer parameters than vector autoregressive (VAR) models, offers a parsimonious, yet effective, alternative for high-dimensional time series. Empirical results demonstrate that MAR models estimated via the proposed methods achieve a comparable fit to VAR models across metrics such as MAE and RMSE. These findings underscore the utility of Yule-Walker and Burg-type estimators in constructing efficient and interpretable models for complex temporal data.
\end{abstract}


\section{Introduction}
In modern time series analysis, increasingly complex and higher-dimensional datasets are often encountered. To address the challenges associated with modeling multivariate time series, the matrix-valued autoregressive (MAR) model has been extensively studied, as seen in \cite{matrix_ar} and \cite{multi_tensor}. This modeling framework has become an active area of research. Least squares estimators, maximum likelihood estimators, and both simulation and real-data studies have been conducted by \cite{samadi2025matrix}. Furthermore, new classes of models have been introduced, such as the matrix autoregressive moving average (MARMA) model proposed by \cite{tsay2024matrix} and the additive autoregressive model (Add-ARM) introduced by \cite{zhang2024additive}.

In this article, we aim to advance this field by proposing alternative estimation methods, such as Yule-Walker equations and Burg's method, for the MAR model. The methods described by \cite{matrix_ar}, including maximum likelihood estimation (MLE) and least squares estimator (LSE), may, in certain cases, fail to preserve the causality of the model. This behavior is particularly undesirable in applications such as financial modeling, where it is critical to assume that future information does not influence present values.
Moreover, in some modeling tasks, it is desirable to preserve the covariance structure of the time series. In such cases, method-of-moments approaches—such as the Yule-Walker and Burg estimators—may be beneficial.

A matrix-valued autoregressive model, denoted as $\MAR(1)$, for the time series $\{\mathbf{X}_t; t \in \ZZ \}$ is defined by the equation:  
$$
\mathbf{X}_t = \mathbf{A} \mathbf{X}_{t-1} \mathbf{B}^T + \mathbf{Z}_t,
$$
where $\mathbf{X}_t: \Omega \rightarrow \RR^{m \times n}$, $\mathbf{A} \in \RR^{m \times m}$, and $\mathbf{B} \in \RR^{n \times n}$ are coefficient matrices, and $\mathbf{Z}_t: \Omega \rightarrow \RR^{m \times n}$ represents matrix-valued white noise. To obtain a uniquely defined MAR model, we normalize the first coefficient by setting $||\mathbf{A}||_F = 1$. This is justified by the fact that, for any scalar $c \neq 0$, we have $(c\mathbf{A})\mathbf{X}_{t-1}(\frac{1}{c}\mathbf{B}^T) = \mathbf{A} \mathbf{X}_{t-1} \mathbf{B}^T$ which shows that the model is invariant under this rescaling of $\mathbf{A}$ and $\mathbf{B}$. We assume that the matrix-valued white noise $\mathbf{Z}_t$ has zero mean and covariance 
$\boldsymbol{\Sigma} = \COV(\vect(\mathbf{Z}_t), \vect(\mathbf{Z}_t)) = \EE[\vect(\mathbf{Z}_t) \vect(\mathbf{Z}_t)^T]$ (see Definition~\ref{def:white_noise}). 

The $\MAR(1)$ model can be associated with the vector autoregressive model $\VVAR(1)$, which is defined by the equation:  
$$
\vect(\mathbf{X}_t) = (\mathbf{B} \otimes \mathbf{A}) \vect(\mathbf{X}_{t-1}) + \vect(\mathbf{Z}_t),
$$  
where $\vect(\cdot)$ denotes the vectorization operation, and $\otimes$ represents the Kronecker product (see Theorem~\ref{twr:three_matrix_multiplication}). This vectorized representation is particularly useful for establishing certain properties of the MAR model. Moreover, the stationarity of the matrix-valued model is interpreted as the stationarity of the corresponding vectorized time series. Thus, the $\MAR(1)$ model is stationary and causal if 
$$\rho(\mathbf{B} \otimes \mathbf{A}) = \rho(\mathbf{A})\cdot \rho(\mathbf{B}) < 1,$$ 
where $\rho$ denotes the spectral radius of a matrix (Definition~\ref{def:sectral_radious}), see~\cite{matrix_ar}.

The MAR model offers two significant advantages: it preserves the original matrix structure—allowing the two coefficient matrices to retain their respective interpretations—and it substantially reduces the number of parameters in the model. Specifically, a $\MAR(1)$ model requires only $m^2 + n^2$ parameters, whereas a standard $\VVAR(1)$ model requires $m^2n^2$ parameters. As demonstrated in Section~\ref{sec:numerical_results}, the $\MAR(1)$ model can provide a similarly good fit compared to the more complex $\VVAR(1)$ model.

Three estimation methods were introduced by \cite{matrix_ar}: the projection method, iterated least squares (LSE), and maximum likelihood estimation (MLE). In this article, we aim to develop new methods, as none of the previously proposed approaches guarantee a causal MAR model or preserve the autocovariance structure of the time series. Notably, no method-of-moments approach has been introduced yet. Therefore, we propose Yule-Walker and Burg-type estimation methods. 

\section{Yule-Walker estimation}
\label{sec:YW}

A reliable method for estimating stationary time series is through the use of Yule-Walker equations. These equations are particularly well-suited for cases where the time series is known to be stationary, as they guarantee a causal model. The Yule-Walker estimation method is computationally efficient when analyzing long time-series datasets. Additionally, it does not impose strict distributional assumptions on the data, allowing for broader applicability across various fields. 

In this section, we derive the solution to the Yule-Walker type equation for the MAR model.
Consider the $\MAR(1)$ model defined by the equation
$$\mathbf{X}_t = \mathbf{A} \mathbf{X}_{t-1} \mathbf{B}^T + \mathbf{Z}_t.$$
By applying the Kronecker product of $\mathbf{X}_{t-1}^T$ to both sides of the above equation, we obtain 
\begin{equation}
    \label{eq:ar_kronecker_mult}
    \mathbf{X}_t \otimes \mathbf{X}_{t-1}^T = \mathbf{A} \mathbf{X}_{t-1} \mathbf{B}^T \otimes \mathbf{X}_{t-1}^T + \mathbf{Z}_t \otimes \mathbf{X}_{t-1}^T.
\end{equation}
Now, consider the first term on the right-hand side and notice that 
$$\mathbf{A} \mathbf{X}_{t-1} \mathbf{B}^T \otimes \mathbf{X}_{t-1}^T = \mathbf{A} \mathbf{X}_{t-1} \mathbf{B}^T \otimes (\mathbf{X}_{t-1}^T I_{m}),$$
therefore, we can use the property of Theorem~\ref{twr:kronecker_prop_mulitplication}. We have 
\begin{align*}
    (\mathbf{A} \mathbf{X}_{t-1} \mathbf{B}^T) \otimes (\mathbf{X}_{t-1}^T I_{m}) &= 
    ((\mathbf{A} \mathbf{X}_{t-1}) \otimes \mathbf{X}_{t-1}^T) (\mathbf{B}^T \otimes I_{m}) \\
    &= ((\mathbf{A} \mathbf{X}_{t-1}) \otimes (I_{n} \mathbf{X}_{t-1}^T)) (\mathbf{B}^T \otimes I_{m}) \\
    &= (\mathbf{A} \otimes I_{n})(\mathbf{X}_{t-1} \otimes \mathbf{X}_{t-1}^T) (\mathbf{B}^T \otimes I_{m}) \\
    &= \mathbf{A}^\mathbf{B} \mathbf{X}_{t-1} \otimes \mathbf{X}_{t-1}^T \mathbf{B}^B,
\end{align*}
where $\mathbf{A}^B = \mathbf{A} \otimes I_n \in \RR^{mn \times mn}$ and $\mathbf{B}^B = \mathbf{B}^T \otimes I_m \in \RR^{mn \times mn}$. Thus, equation (\ref{eq:ar_kronecker_mult}) has the form
\begin{equation}
    \label{eq:ar_kronecker_mult2}
    \mathbf{X}_t \otimes \mathbf{X}_{t-1}^T = \mathbf{A}^B \mathbf{X}_{t-1} \otimes \mathbf{X}_{t-1}^T \mathbf{B}^B + \mathbf{Z}_t \otimes \mathbf{X}_{t-1}^T.
\end{equation}
We apply the expected value to both sides of the equation (\ref{eq:ar_kronecker_mult2}). We have
\begin{align*}
    \EE[\mathbf{X}_t \otimes \mathbf{X}_{t-1}^T] &= \mathbf{A}^B \EE[\mathbf{X}_{t-1} \otimes \mathbf{X}_{t-1}^T] \mathbf{B}^B + \EE[\mathbf{Z}_t \otimes \mathbf{X}_{t-1}^T] \\
    \Gamma_1^\otimes &= \mathbf{A}^B \Gamma_0^\otimes \mathbf{B}^B + \EE[\mathbf{Z}_t \otimes \mathbf{X}_{t-1}^T],
\end{align*}
where $\Gamma_i^\otimes = \EE[\mathbf{X}_t \otimes \mathbf{X}_{t-i}^T]$ and $\EE[\mathbf{Z}_t \otimes \mathbf{X}_{t-1}^T] = 0_{mn \times mn}$ from the assumption that 
$\mathbf{Z}_t$ is matrix-valued white noise. Thus,
\begin{equation}
\label{eq:YW_mar_1}
    \Gamma_1^\otimes = \mathbf{A}^B \Gamma_0^\otimes \mathbf{B}^B.
\end{equation}
The Yule-Walker equation for matrix-valued $\MAR(1)$ has the form
\begin{equation*}
    \begin{cases}
        \Gamma_1^\otimes = \mathbf{A}^B \Gamma_0^\otimes \mathbf{B}^B \\
        \boldsymbol{\Sigma}^\otimes = \Gamma_0^\otimes - \mathbf{A}^B \Gamma_{-1}^\otimes \mathbf{B}^B.
    \end{cases}
\end{equation*}

Equations can also be derived for the $\MAR(p)$ model. By applying the Kronecker product of $\mathbf{X}_{t-1}^T, \mathbf{X}_{t-2}^T, \ldots, \mathbf{X}_{t-p}^T$ to the model equation and following the same approach, we obtain  
\begin{equation*}
    \begin{cases}
        \Gamma_1^\otimes = \mathbf{A}_1^B \Gamma_0^\otimes \mathbf{B}_1^B +  \mathbf{A}_2^B \Gamma_{-1}^\otimes \mathbf{B}_2^B + \ldots + \mathbf{A}_p^B \Gamma_{1-p}^\otimes \mathbf{B}_p^B \\
        \Gamma_2^\otimes = \mathbf{A}_1^B \Gamma_1^\otimes \mathbf{B}_1^B +  \mathbf{A}_2^B \Gamma_{0}^\otimes \mathbf{B}_2^B + \ldots + \mathbf{A}_p^B \Gamma_{2-p}^\otimes \mathbf{B}_p^B \\
        \vdots \\
        \Gamma_p^\otimes = \mathbf{A}_1^B \Gamma_{p-1}^\otimes \mathbf{B}_1^B +  \mathbf{A}_2^B \Gamma_{p-2}^\otimes \mathbf{B}_2^B + \ldots + \mathbf{A}_p^B \Gamma_{0}^\otimes \mathbf{B}_p^B \\
        \boldsymbol{\Sigma}^\otimes = \Gamma_0^\otimes - \mathbf{A}_1^B \Gamma_{-1}^\otimes \mathbf{B}_1^B - \mathbf{A}_2^B \Gamma_{-2}^\otimes \mathbf{B}_2^B - \ldots - \mathbf{A}_p^B \Gamma_{-p}^\otimes \mathbf{B}_p^B.
    \end{cases}
\end{equation*}
Finding the exact solution to these equations for the estimated matrices $\Gamma_i^{\otimes}$ may not be possible, as the matrices $\Gamma_i^{\otimes}$ may not have the appropriate block structure. However, we can determine the closest solution by minimizing the Frobenius norm. For the $\MAR(1)$ model, this corresponds to the optimization problem 
\begin{equation}
    \label{eq:MAR1_YW_norm}
    \min_{\mathbf{A}_1, \mathbf{B}_1} ||\Gamma_1^\otimes - (\mathbf{A}_1 \otimes I_n) \Gamma_0^\otimes (\mathbf{B}_1^T \otimes I_m) ||_F \quad,
\end{equation}
subject to the constraint $||\mathbf{A}_1||_F = 1$ to ensure uniqueness of the solution. This approach simultaneously solves both the Nearest Kronecker Product problem and the Yule-Walker equations, potentially resulting in a more efficient estimation procedure. However, further insight into the solution method is needed. Currently, our solution to equation (\ref{eq:MAR1_YW_norm}) is based on a Python implementation of the L-BFGS-B algorithm (\cite{LBFGS_alg}). As shown in Section~\ref{sec:YW_numerical_res}, this approach is not computationally efficient.

The matrices $\Gamma_k^\otimes = \EE[\mathbf{X}_t \otimes \mathbf{X}_{t-k}^T]$ contain the same elements as the matrices
$\Gamma_k = \EE[\vect(\mathbf{X}_t) \vect(\mathbf{X}_{t-k})^T]$. For a zero-mean process, we have $\Gamma_k = \COV(\vect(\mathbf{X}_t), \vect(\mathbf{X}_{t-k}))$. This means that $\Gamma_k^\otimes$ can be computed by first calculating the covariance matrix $\Gamma_k$ and then permuting its elements accordingly. It can be shown (see Theorem~\ref{twr:cov_matrix_permutation}) that 
$$\Gamma_k^\otimes = [\gamma_{ij}^\otimes]_{i,j=0,\ldots,mn-1} = [\gamma_{(i \% m)*n + (j\% n) , (i//m) *m + j//n}]_{i,j=0,\ldots,mn-1},$$
where $\Gamma_k = [\gamma_{ij}]_{i,j=0,\ldots,mn-1}$, $\%$ denotes the remainder, and $//$ truncating integer division. Thus, in implementation, we can compute the autocovariance for the vectorized time series $\mathbf{X}_t$ and then adjust the indices to obtain the Kronecker-structured matrix $\Gamma_k^\otimes$.

\section{Burg's method}
\label{sec:Burg}
A second approach to the estimation of $\MAR(p)$ parameters is based on a Burg-type method. This method was originally developed for one-dimensional time series by \cite{burg_orginal}; detailed derivations can be found in \cite{burg_onedim} and \cite{brock_davis}. The multivariate (vector-valued) case was discussed by \cite{burg_multivariate}.
This algorithm is particularly effective for short time series and provides numerically stable results. The estimated AR model satisfies both stationarity and causality conditions, making it a robust choice for modeling time series data. Therefore, we aim to derive Burg's algorithm for matrix-valued time series.

Consider a sample $\mathbf{X}_1, \ldots, \mathbf{X}_N$, and define the forward prediction residual $\mathbf{f}_t^{(k)} \in \RR^{m \times n}$ and the backward prediction residual $\mathbf{b}_t^{(k)} \in \RR^{m \times n}$ for the $\MAR(k)$ model as
\begin{align*}
    \mathbf{f}_t^{(k)} &= \mathbf{X}_t - \hat{\mathbf{X}}_t(k)  \\
    \mathbf{b}_t^{(k)} &= \mathbf{X}_{t-k} -  \hat{\mathbf{X}}^b_{t-k}(k),
\end{align*}
where 
\begin{align*}
    \hat{\mathbf{X}}_t(k) &= \sum\limits_{i=1}^k \mathbf{A}_k(i) \mathbf{X}_{t-i} (\mathbf{B}_k(i))^T \\
    \hat{\mathbf{X}}^\mathbf{b}_t(k) &= \sum\limits_{i=1}^k \mathbf{A}_k(i) \mathbf{X}_{t+i} (\mathbf{B}_k(i))^T
\end{align*}
are, respectively, the best forward and the best backward (in time) linear predictor of $\mathbf{X}_{t}$. The term "best linear predictor of $\mathbf{X}_t$" means that the coefficients $\mathbf{A}_k(1), \ldots, \mathbf{A}_k(k)$ and $\mathbf{B}_k(1), \ldots, \mathbf{B}_k(k)$ minimize the average squared error
$$\EE\left[||\mathbf{X}_t - \hat{\mathbf{X}}_t(k)||^2_F\right],$$
where $||\cdot||_F$ denotes the Frobenius norm. Note that this is equivalent to finding the minimum of
$\EE\left[||\vect(\mathbf{X}_t) - \vect(\hat{\mathbf{X}}_t(k))||^2\right]$ under the Euclidean norm $||\cdot||$.

Burg's method aims to find the coefficients $\mathbf{A}_k(k)$ and $\mathbf{B}_k(k)$ that minimize the total prediction error
\begin{equation}
    \label{eq:prediction_error_E_k}
    E^{(k)} = \sum\limits_{t = k+1}^N ||\mathbf{f}_t^{(k)}||^2_F + ||\mathbf{b}_t^{(k)}||^2_F.
\end{equation}
To solve this problem, we first need to express $E^{(k)}$ in terms of $\mathbf{A}_k(k)$, $\mathbf{B}_k(k)$, and the coefficients known from step $k-1$. Therefore, it is necessary to establish the relationship between $\mathbf{f}_t^{(k)}$, $\mathbf{b}_t^{(k)}$, and the corresponding residuals $\mathbf{f}_t^{(k-1)}$ and $\mathbf{b}_t^{(k-1)}$.

Consider the associated $\VVAR(k)$ model, defined by the equation
\begin{equation*}
\label{eq:VAR_p}
    \vect(\mathbf{X}_t) = (\mathbf{B}_1 \otimes \mathbf{A}_1) \vect(\mathbf{X}_{t-1}) + \ldots + (\mathbf{B}_k \otimes \mathbf{A}_k) \vect(\mathbf{X}_{t-k}) + \vect(\mathbf{Z}_t).
\end{equation*}
In this case, the equations for updating the forward and backward errors take the form 
\begin{align*}
    \vect(\mathbf{f}_t^{(k)}) &= \vect(\mathbf{f}_t^{(k-1)}) - (\mathbf{B}_k(k) \otimes \mathbf{A}_k(k)) \vect(\mathbf{b}_{t-1}^{(k-1)}) \\
    \vect(\mathbf{b}_t^{(k)}) &= \vect(\mathbf{b}_{t-1}^{(k-1)}) - (\mathbf{B}_k(k) \otimes \mathbf{A}_k(k)) \vect(\mathbf{f}_t^{(k-1)}).
\end{align*}
Thus, after applying the reverse vectorization operation, we obtain
\begin{align}
\label{eq:forward_backward}
\begin{split}
    \mathbf{f}_t^{(k)} &= \mathbf{f}_t^{(k-1)} - \mathbf{A}_k(k) \mathbf{b}_{t-1}^{(k-1)} (\mathbf{B}_k(k))^T \\
    \mathbf{b}_t^{(k)} &= \mathbf{b}_{t-1}^{(k-1)} - \mathbf{A}_k(k) \mathbf{f}_t^{(k-1)} (\mathbf{B}_k(k))^T.
\end{split}
\end{align}
Next, we can substitute equation (\ref{eq:forward_backward}) into equation (\ref{eq:prediction_error_E_k}). To have more managable equations, we use notation $\mathbf{A}_k(k) = \mathbf{A}$, $ \mathbf{B}_k(k) = \mathbf{B}$, $\mathbf{b}_t = \mathbf{b}_t^{(k-1)}$ and $\mathbf{f}_t = \mathbf{f}_t^{(k-1)}$. 
If we fix the matrix $\mathbf{B}$, compute the derivative with respect to $\vect(\mathbf{A})$, and set it equal to zero, we obtain 
\begin{equation}
\label{eq:burg_solve_A}
    \mathbf{A} =  
    \sum_{t=k+1}^{N} \left[
     \mathbf{f}_t\mathbf{B}\mathbf{b}_{t-1}^T +   \mathbf{b}_{t-1}\mathbf{B} \mathbf{f}_t^T\right]
     \left( \sum_{t=k+1}^{N} \left[ \left(\mathbf{B} \mathbf{b}_{t-1}^T\right)^2 +  \left(\mathbf{B}\mathbf{f}_t^T\right)^2 \right] \right)^{-1}
\end{equation}
where the square of a matrix $\mathbf{B}$ is defined as $\mathbf{B}^2 := \mathbf{B}^T \mathbf{B}$. Next, if we fix the matrix $\mathbf{A}$, compute the derivative with respect to $\vect(\mathbf{B}^T)$, and set it equal to zero, we obtain 
\begin{equation}
\label{eq:burg_solve_B}
    \mathbf{B}^T = \left( \sum_{t=k+1}^{N} \left[ \left(\mathbf{A}\mathbf{b}_{t-1}\right)^2 -  \left(\mathbf{A} \mathbf{f}_t \right)^2 \right] \right)^{-1}
    \sum_{t=k+1}^{N} \left[
     \mathbf{b}_{t-1}^T\mathbf{A}^T \mathbf{f}_t + \mathbf{f}_t^T\mathbf{A}^T \mathbf{b}_{t-1} \right].
\end{equation}
More details on the computation of these solutions can be found in Theorem~\ref{twr:min_Ek}. The inverse matrices in the equations above may not always exist; however, we can consider a generalized matrix inverse, such as the Moore–Penrose inverse. Using this type of inverse yields the minimum norm (or least squares) solution, see~\cite{general_inv}. 
To solve the equations, we initialize the algorithm by selecting random matrices $\mathbf{A}$ and $\mathbf{B}$, and then iteratively update both matrices using equations (\ref{eq:burg_solve_A}) and (\ref{eq:burg_solve_B}). At each step, the matrix $\mathbf{A}$ should also be normalized by setting $\mathbf{A} := \frac{\mathbf{A}}{||\mathbf{A}||_F}$ to ensure that the solution remains unique.

Furthermore, in each step, we need to compute the parameters $\mathbf{A}_k(1), \ldots, \mathbf{A}_k(k-1)$ and $\mathbf{B}_k(1), \ldots, \mathbf{B}_k(k-1)$. For the standard $\VVAR(k)$ model described by equation (\ref{eq:VAR_p}), we have
\begin{equation}
\label{eq:update_AB}
    \mathbf{B}_k(i) \otimes \mathbf{A}_k(i) = \mathbf{B}_{k-1}(i) \otimes \mathbf{A}_{k-1}(i) - \left[\mathbf{B}_k(k) \otimes \mathbf{A}_k(k)\right] \left[\mathbf{B}_{k-1}(k-i) \otimes \mathbf{A}_{k-1}(k-i)\right],
\end{equation}
for $i=1,\ldots,k-1$. The parameters $\mathbf{A}_k(i)$ and $\mathbf{B}_k(i)$ can be obtained by solving the Nearest Kronecker Product problem for the equation above, i.e.,
$$\min_{\mathbf{A}_k(i), \mathbf{B}_k(i)} ||\mathbf{X} - \mathbf{B}_k(i) \otimes \mathbf{A}_k(i)||_F \quad,$$
where $\mathbf{X}$ corresponds to the right-hand side of equation (\ref{eq:update_AB}).

The algorithm for solving the $\MAR(p)$ model can be described in the following steps: 
\begin{enumerate}[Step 1:] 
    \item Initialize the algorithm by setting $\mathbf{f}_t^{(0)} = \mathbf{X}_t$ and $\mathbf{b}_t^{(0)} = \mathbf{X}_t$, and repeat the following steps for $k = 1, \ldots, p$. 
    \item Compute $\mathbf{A}_k(k)$ and $\mathbf{B}_k(k)$ using formulas (\ref{eq:burg_solve_A}) and (\ref{eq:burg_solve_B}). 
    \item Calculate the parameters $\mathbf{A}_k(1), \ldots, \mathbf{A}_k(k-1), \mathbf{B}_k(1), \ldots, \mathbf{B}_k(k-1)$ using the formula (\ref{eq:update_AB}).
    \item Update $\mathbf{f}_t^{(k)}$ and $\mathbf{b}_t^{(k)}$ using formula (\ref{eq:forward_backward}).
\end{enumerate}
The output of the algorithm consists of the matrices $\mathbf{A}_p(1), \ldots, \mathbf{A}_p(p)$ and $\mathbf{B}_p(1), \ldots, \mathbf{B}_p(p)$, which serve as the parameters of the $\MAR(p)$ model.

\section{Numerical results}
\label{sec:numerical_results}
In this section, we present the results of benchmarking various estimation approaches. We compare the Yule-Walker and Burg methods for the $\MAR(1)$ model, as described in the preceding sections, with the LSE method introduced by \cite{matrix_ar}. Additionally, we evaluate the approach of estimating the corresponding $\VVAR(1)$ model and solving the Nearest Kronecker Product (NKP) problem to obtain the coefficient matrices $\mathbf{A}$ and $\mathbf{B}$.
\subsection{Testing methodology}
The results are based on artificially generated data. For each model, given matrix size (only square matrices were considered) and time series length, 100 vector-valued time series were generated using the $\VVAR(1)$ equation. The parameters of the $\VVAR(1)$ model were initialized randomly and then adjusted using the following approach.

The matrix $\boldsymbol{\phi}^{(r)} \in \RR^{m \times m}$ is randomly drawn from a standard normal distribution. To ensure the stationarity of the generated time series, i.e., $\rho(\boldsymbol{\phi}) < 1$, we rescale it as follows:
$$\boldsymbol{\phi} := \frac{1}{\rho(\boldsymbol{\phi}^{(r)}) + 1} \boldsymbol{\phi}^{(r)}.$$
Next, we construct the covariance matrix $\boldsymbol{\Sigma} \in \RR^{m \times m}$ for the white noise process. A matrix $\mathbf{S} \in \RR^{m \times m}$ is randomly sampled from a standard normal distribution. To ensure that $\boldsymbol{\Sigma}$ is symmetric, we compute
$$\mathbf{S}^s = \frac{1}{2}(\mathbf{S} + \mathbf{S}^T).$$
To make $\boldsymbol{\Sigma}$ positive semi-definite, we perform an eigendecomposition:
$$\mathbf{S}^s = Q \Lambda Q^T, $$
where $\Lambda = \diag(\lambda_i)$. We then define $\Lambda^+ := \diag(|\lambda_i|)$ and set
$$\boldsymbol{\Sigma} := Q \Lambda^+ Q^T.$$
Finally, we generate the vector time series $X_1, \ldots, X_N$ using the equation
$$X_t = \phi X_{t-1} + \epsilon_t, $$
where $\epsilon_t$ are realizations of white noise with distribution $\mathcal{N}(0, \boldsymbol{\Sigma})$.

The comparison of models was conducted across time series lengths ranging from 100 to 500 and matrix sizes from 2 to 10. Each model was fitted using a generated sample and evaluated on the subsequent 100 generated observations. For example, for a time series of length 100, a series of length 200 was generated: the first 100 observations were used for training, and the remaining 100 were used for testing the model using the selected evaluation metrics.

For each iteration, the average execution time was recorded. Additionally, for each run, standard model performance metrics, such as MAE, SMAPE, and RMSE, were computed based on the test data. The exact formula for the metrics used can be found in Definition~\ref{def:metrics}. In parallel, Mardia's test for joint multivariate normality of the residuals was conducted (\cite{mardia1970measures}). Due to technical limitations of the implemented test ("multivariate\_normality" from the Python package "pingouin"), the calculations for $10 \times 10$ matrix-valued time series could not be performed and were therefore not included in Table~\ref{tab:mardia_p_values}.
\begin{table}[H]
\centering
\caption{P-values of Mardia's tests for model residuals}
\label{tab:mardia_p_values}
\begin{tabular}{lrrrr}
\toprule
Model & Minimum & Quantile 1\% & Quantile 5\% & Median \\
\midrule
MAR(1) Burg & 0.0000 & 0.0093 & 0.0958 & 0.3461 \\
MAR(1) LSE & 0.0000 & 0.0131 & 0.0852 & 0.3457 \\
MAR(1) Yule-Walker & 0.0000 & 0.0152 & 0.0726 & 0.3454 \\
VAR(1) Burg & 0.0002 & 0.0179 & 0.0983 & 0.3475 \\
VAR(1) Yule-Walker & 0.0000 & 0.0113 & 0.0849 & 0.3427 \\
vec MAR(1) Burg & 0.0000 & 0.0136 & 0.0886 & 0.3446 \\
vec MAR(1) Yule-Walker & 0.0000 & 0.0139 & 0.0867 & 0.3422 \\
\bottomrule
\end{tabular}
\end{table}

\subsection{Yule-Walker}
\label{sec:YW_numerical_res}
We compared the Yule-Walker method described in Section~\ref{sec:YW} with the standard $\VVAR(1)$ model estimated using the classical vector Yule-Walker equations, and with the vectorized $\MAR(1)$ model, where the $\AR(1)$ model was first estimated and then the Nearest Kronecker Product (NKP) problem was solved to obtain the matrices $\mathbf{A}$ and $\mathbf{B}$.

We observe that the MAE and RMSE metrics are nearly identical across all data samples (Figure~\ref{fig:YW_d300} and Figure~\ref{fig:YW_m5}). This suggests that the estimation method produces errors of consistent magnitude, without significant skewness, heavy tails, or heteroskedasticity. Moreover, the Mardia's tests (Table~\ref{tab:mardia_p_values}) confirm the normality of the model residuals in nearly all cases, indicating a good model fit.

The RMSE and SMAPE values are similar overall; however, on average, the SMAPE is lower for the standard $\VVAR(1)$ model, while the RMSE is lower for the $\MAR(1)$ model. This comparable predictive performance indicates that the MAR-based approach is well suited for high-dimensional datasets, as it achieves similar accuracy while using significantly fewer parameters.

Regarding execution time, the Yule-Walker method for the $\MAR(1)$ model is slower than the benchmark methods (vectorized $\MAR(1)$ and $\VVAR(1)$) in high-dimensional settings. However, for low-dimensional but longer time series, its performance is comparable to that of the other two methods.

\begin{figure}[H]
    \centering
    \includegraphics[width=.9\linewidth]{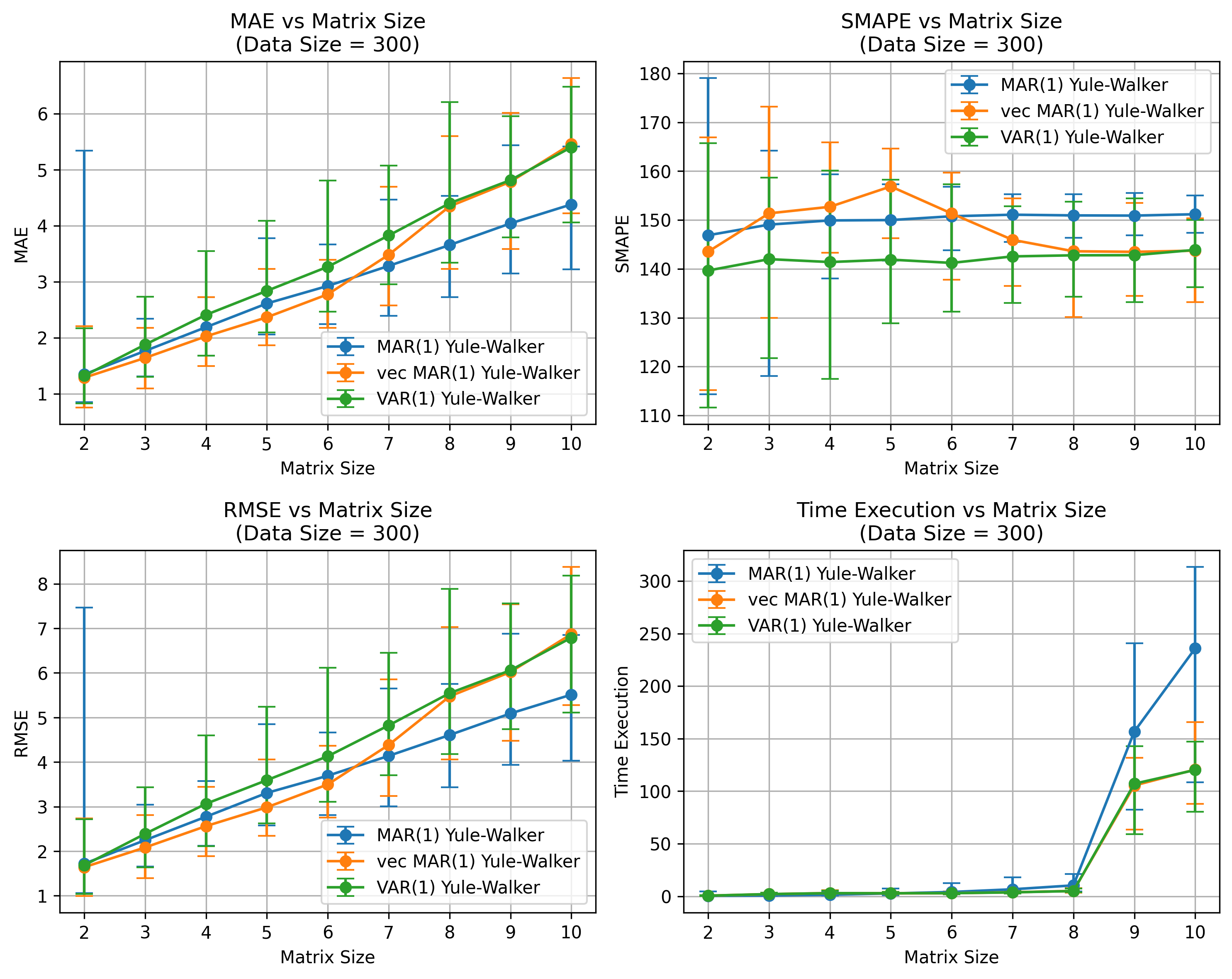}
    \caption{Comparison of the Yule-Walker approach for the $\MAR(1)$ and $\VVAR(1)$ models for a time series of length 300.}
    \label{fig:YW_d300}
\end{figure}
\begin{figure}[H]
    \centering
    \includegraphics[width=.9\linewidth]{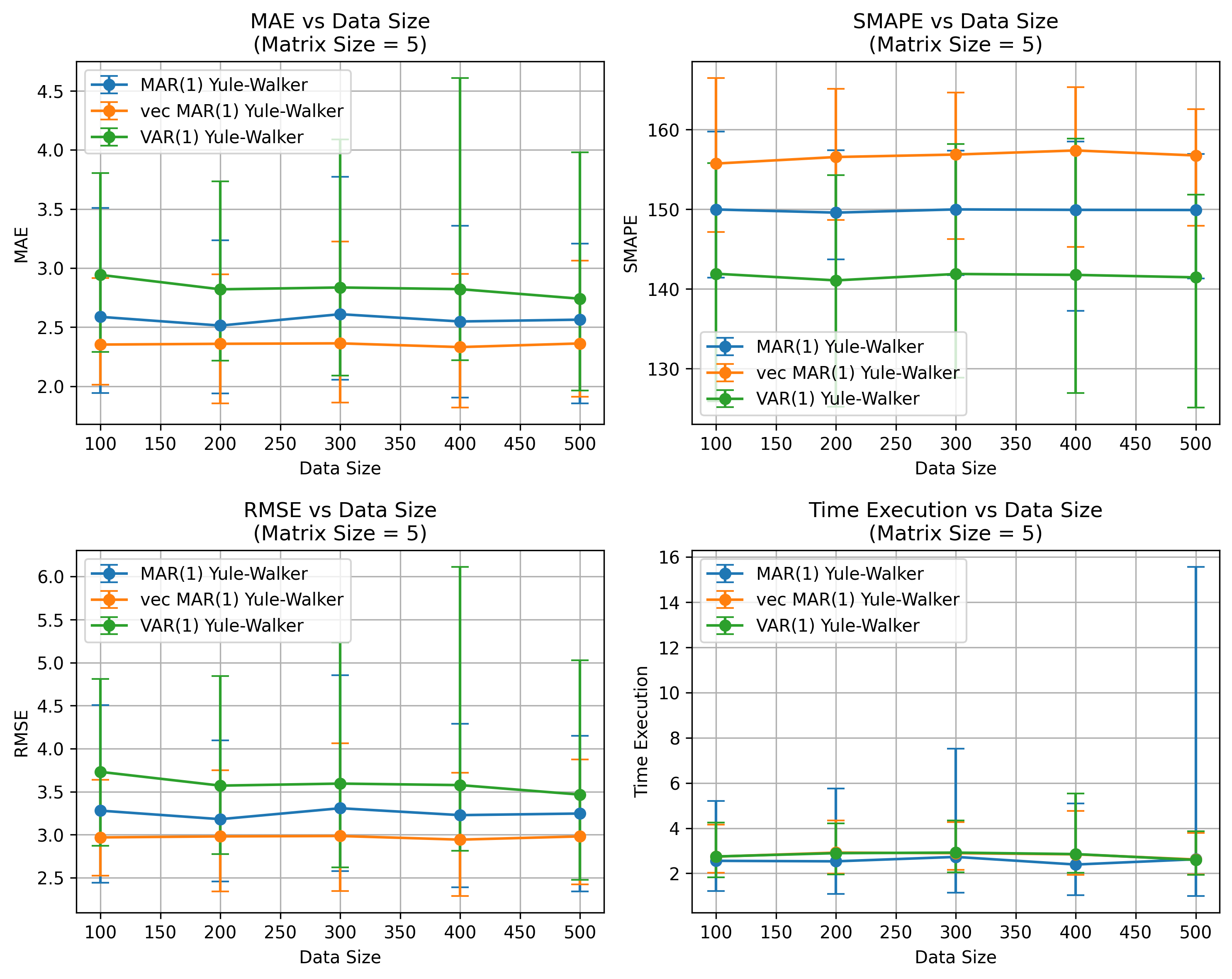}
    \caption{Comparison of the Yule-Walker approach for the $\MAR(1)$ and $\VVAR(1)$ models with a matrix size of 5.}
    \label{fig:YW_m5}
\end{figure}
\subsection{Burg's method}
Similarly, the Burg method described in Section~\ref{sec:Burg} was compared with the Burg method applied to the classical $\VVAR(1)$ model and the vectorized $\MAR(1)$ model.

Similar to the Yule-Walker method, the MAE and RMSE metrics yield nearly identical results, and the Mardia's tests (Table~\ref{tab:mardia_p_values}) indicates a good model fit. Some exceptions are observed for the $\VVAR(1)$ model; however, 95\% of the samples still have p-values greater than 0.175. In terms of the SMAPE metric, the $\VVAR(1)$ model performs noticeably better; however, the lower RMSE for the $\MAR(1)$ model suggests more accurate point predictions (Figure~\ref{fig:burg_d300} and Figure~\ref{fig:burg_m5}).

The execution time for both the $\MAR(1)$ and vectorized $\MAR(1)$ methods is considerably longer than for the classical $\VVAR(1)$ model, primarily due to the extensive matrix multiplications involved in the estimation process. This overhead could likely be mitigated with more efficient code implementation.

\begin{figure}[H]
    \centering
    \includegraphics[width=.9\linewidth]{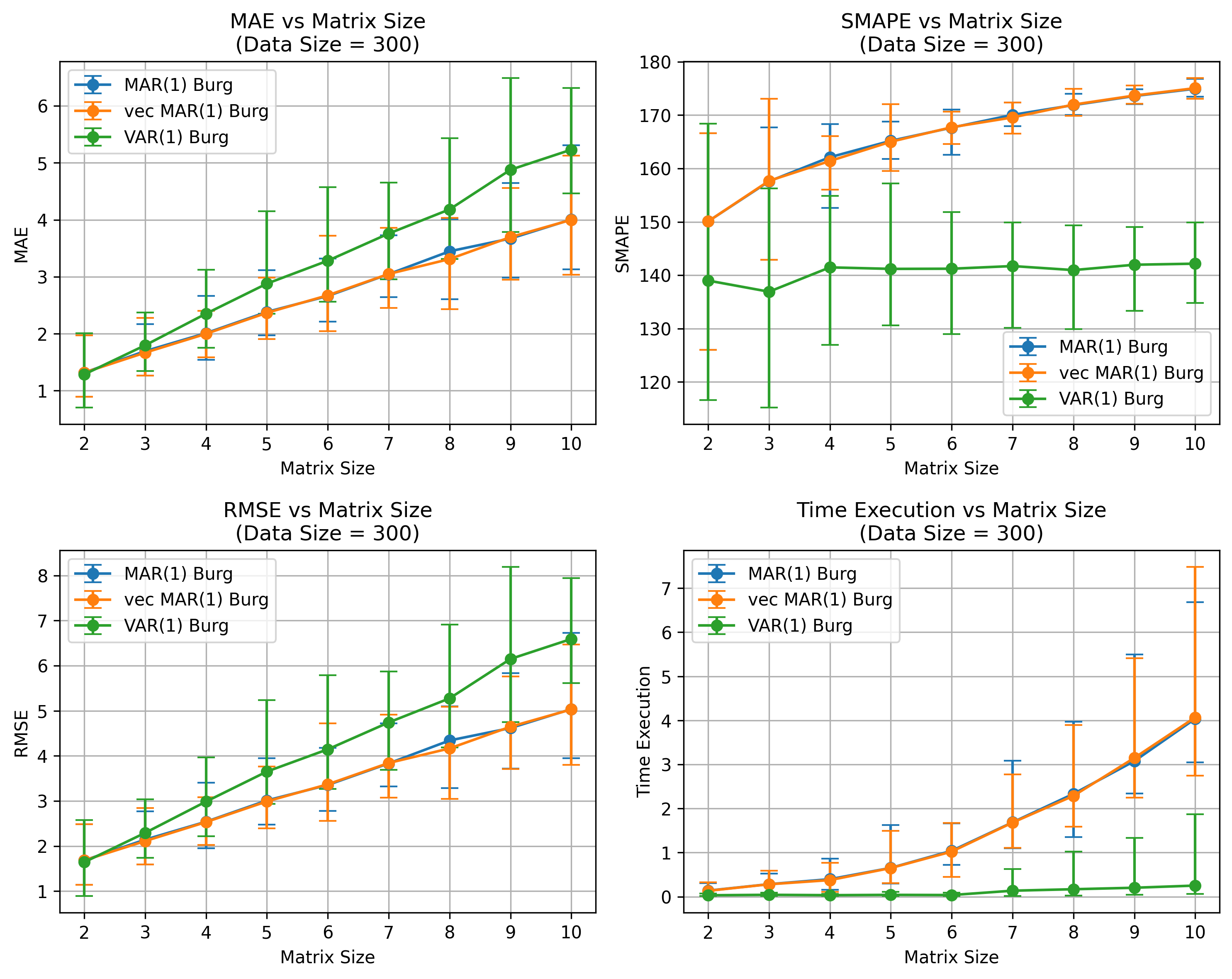}
    \caption{Comparison of the Burg approach for the $\MAR(1)$ and $\VVAR(1)$ models for a time series of length 300. For metrics charts, the blue and orange lines overlap.}
    \label{fig:burg_d300}
\end{figure}

\begin{figure}[H]
    \centering
    \includegraphics[width=.9\linewidth]{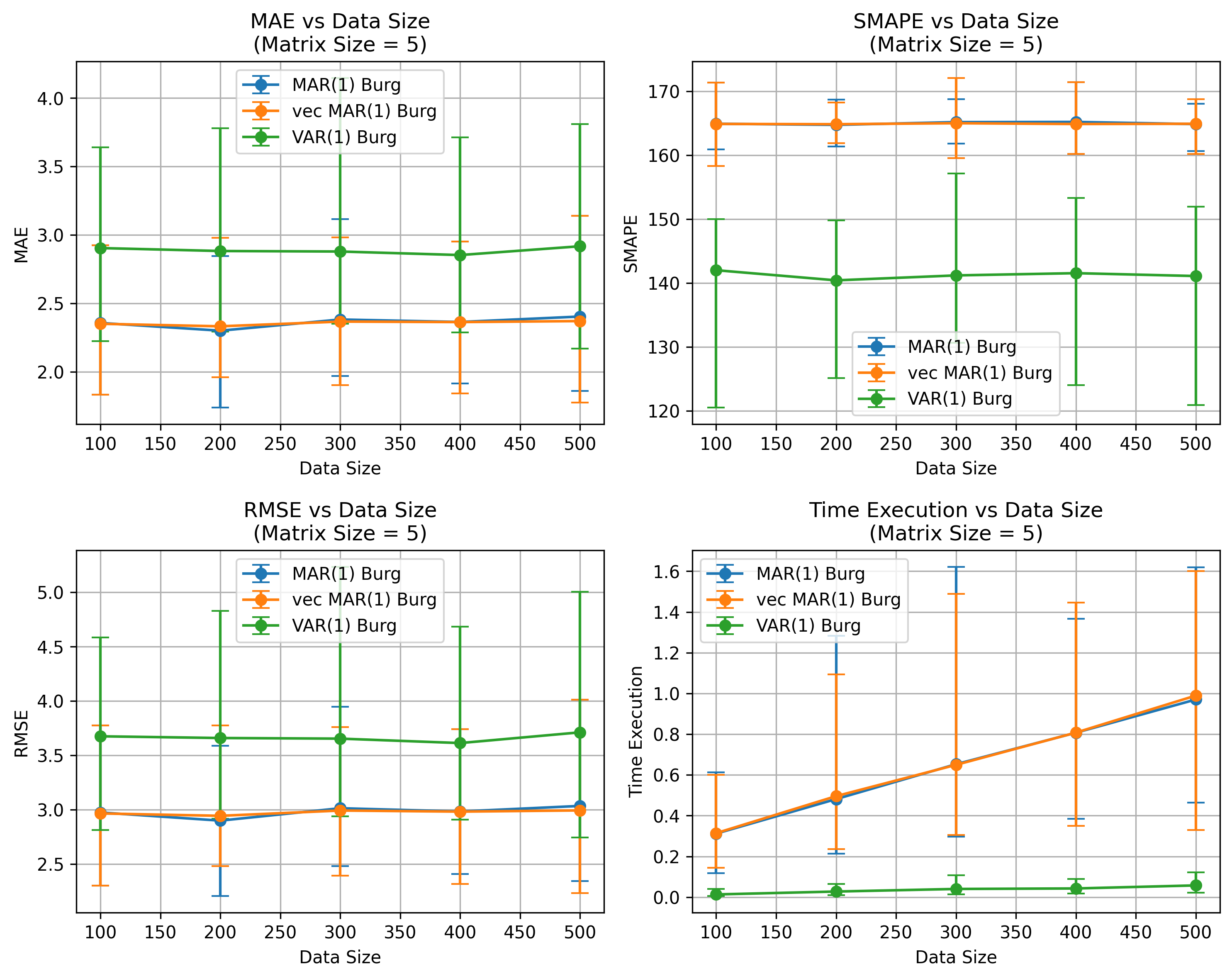}
    \caption{Comparison of the Burg approach for the $\MAR(1)$ and $\VVAR(1)$ models with a matrix size of 5. For metrics charts, the blue and orange lines overlap.}
    \label{fig:burg_m5}
\end{figure}

\subsection{Comparison of YW, Burg and LSE}
The comparison of the methods described in this article with the LSE method introduced by \cite{matrix_ar} shows that all three approaches yield similar MAE and RMSE values. In terms of SMAPE, the Yule-Walker and LSE methods perform better (Figure~\ref{fig:als_d300} and Figure~\ref{fig:als_m5}).

Regarding execution time, the LSE method is faster for short but high-dimensional time series; however, its performance deteriorates as the length of the time series increases. The Burg method is consistently fast across all scenarios, while the Yule-Walker method performs exponentially worse as the dimensionality of the data increases. The consistency and stability of the Burg method are clearly illustrated in Figure~\ref{fig:lse_burg_surface}, where the surface plot for Burg appears significantly smoother.

\begin{figure}[H]
    \centering
    \includegraphics[width=.9\linewidth]{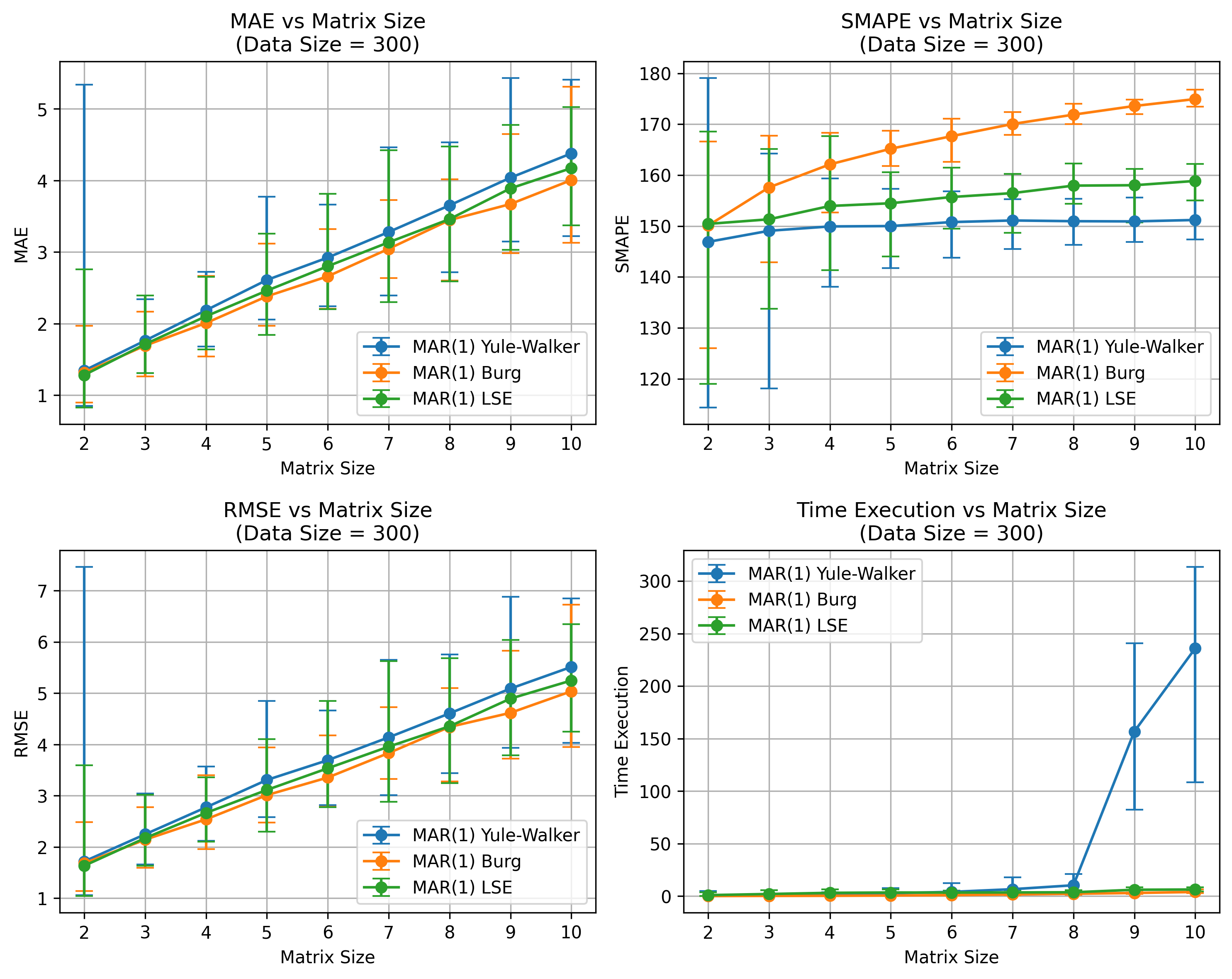}
    \caption{Comparison of the YW, Burg and LSE approach for the $\MAR(1)$ model for a time series of length 300. For execution time chart, the red and orange lines overlap.}
    \label{fig:als_d300}
\end{figure}

\begin{figure}[H]
    \centering
    \includegraphics[width=.9\linewidth]{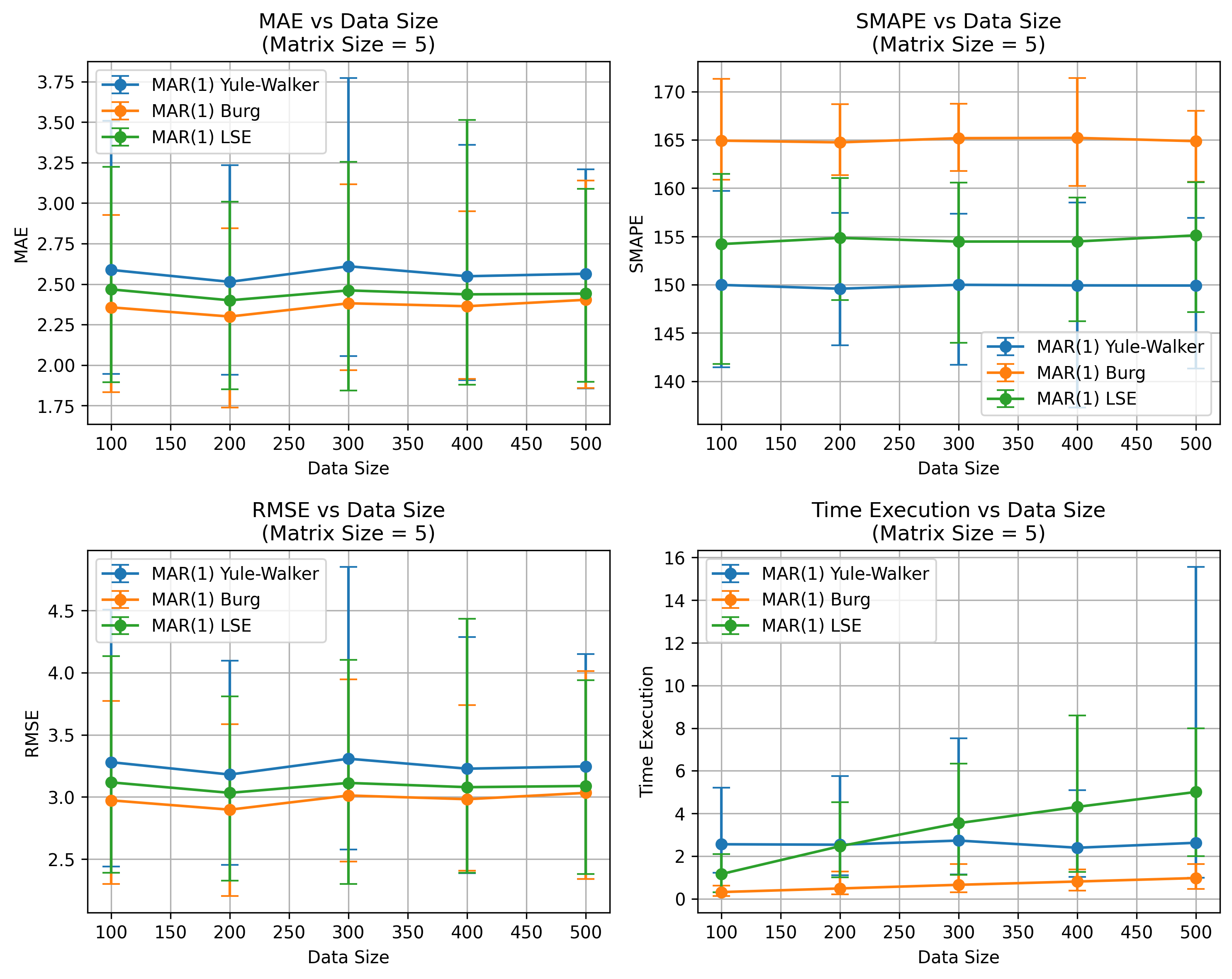}
    \caption{Comparison of the YW, Burg and LSE approach for the $\MAR(1)$ model with a matrix size of 5.}
    \label{fig:als_m5}
\end{figure}

\begin{figure}[H]
    \centering
    \includegraphics[width=.9\linewidth]{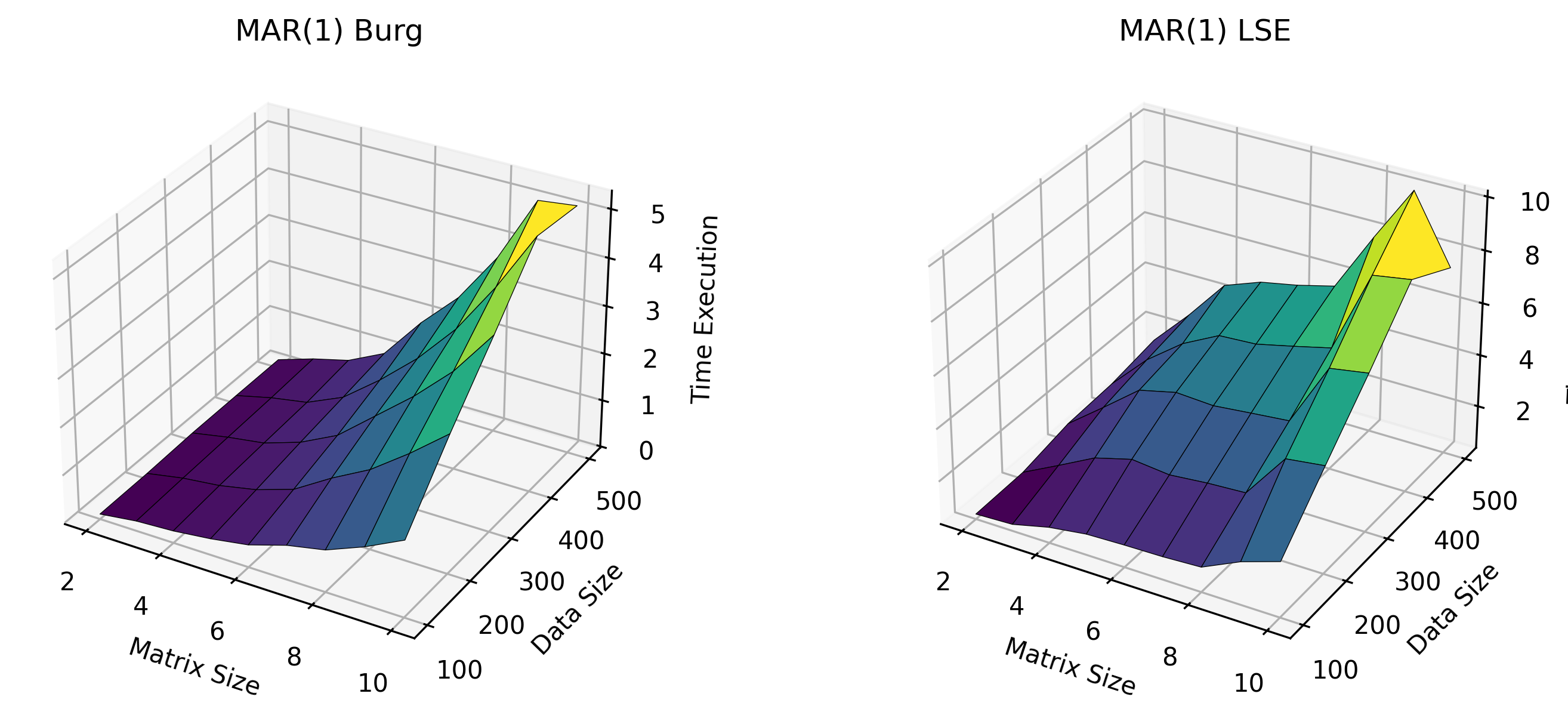}
    \caption{Comparison of the execution time of the Burg and LSE approach for the $\MAR(1)$ }
    \label{fig:lse_burg_surface}
\end{figure}

\section{Conclusions}
MAR models provide a similarly good fit as VAR models but with fewer parameters, making them preferable when we do not want to overfit the model.

The MAR-based approach is better suited for large-dimensional datasets, as it involves significantly fewer parameters while maintaining a comparable quality of fit to the VAR-based approach. For some metrics, such as MAE and RMSE, matrix-valued models can even yield better results.

The Burg and Yule-Walker (YW) methods show an advantage over least squares estimation (LSE) when the time series is long. Moreover, the Burg and YW methods preserve the autocorrelation structure better than maximum likelihood estimation (MLE) or LSE methods, as these estimators are directly based on autocovariance matrices.

\section*{Acknowledgements}
This article has been completed while the author was the Doctoral Candidate in the Interdisciplinary Doctoral School at the Lodz University of Technology, Poland. 

\clearpage

\appendix
\section{Appendix - known facts and definitions}
This section presents general facts, standard definitions, and key properties of the Kronecker product.

\renewcommand{\thetheorem}{A\arabic{theorem}}
\renewcommand{\thedefinition}{A\arabic{definition}}

\setcounter{theorem}{0}

\begin{definition}
\label{def:sectral_radious}
    Let $\lambda_1, \ldots, \lambda_n$ be the eigenvalues of a matrix $A \in \RR^{n \times n}$. \textbf{Spectral radius} of $A$ is defined as
    $$\rho(A) = \max(|\lambda_1|, \ldots, |\lambda_n|).$$
\end{definition}

\begin{definition}
\label{def:vectorization}
    The \textbf{vectorization} of matrix $A \in \RR^{m \times n}$ is a linear transformation of matrix $A = [a_{ij}]_{\substack{i = 1,\ldots,m\\ j = 1,\ldots,n}}$ to a vector i.e. 
    $$\vect(A) = [ a_{1,1}, \ldots, a_{m,1}, a_{1,2}, \ldots, a_{m,2}, \ldots, a_{1,n}, \ldots, a_{m,n}]^T.$$ 
\end{definition}

\begin{theorem}
\label{twr:kronecker_prop_mulitplication}
    If $A, B, C$, and $D$ are matrices of appropriate dimensions to allow the multiplication of $AC$ and $BD$, then

    $$(A \otimes B) (C \otimes D) = (AC) \otimes (BD).$$
\end{theorem}

\begin{theorem}
\label{twr:three_matrix_multiplication}
    Let $A \in \RR^{k \times l},\, B \in \RR^{l \times m}$ and $C \in \RR^{m \times n}$. The multiplication 
    of three matrices can be expressed in terms of Kronecker product multiplied by a vectorized matrix
    $$\vect(ABC) = (C^T \otimes A) \vect(B).$$
\end{theorem}

\begin{definition}
\label{def:white_noise}
    The matrix-valued time series $\{\mathbf{Z}_t; t \in \ZZ \}$ is called
    \textbf{matrix-valued white noise} if it is stationary with mean $0{m \times n}$ and has an autocovariance function given by
    \begin{equation*}
        \Gamma_h := \COV(\vect(\mathbf{Z}_{t+h}), \vect(\mathbf{Z}_t)) = 
        \begin{cases}
        \boldsymbol{\Sigma}, \text{ if } h = 0, \\
        0_{mn \times mn}, \text{ otherwise, }
        \end{cases}
    \end{equation*}
    where $\boldsymbol{\Sigma} \in \RR^{mn \times mn}$, and the vectorization operator $\vect(\cdot)$ is defined in Definition~\ref{def:vectorization}.
\end{definition}

\begin{definition}
\label{def:metrics}
    Let $\mathbf{X}_1, \ldots, \mathbf{X}_N$ be an $m \times n$ matrix-valued time series, and let $\hat{\mathbf{Y}}_1, \ldots, \hat{\mathbf{Y}}_N$ denote the corresponding predicted values from a given model. We define the \textbf{Mean Absolute Error (MAE)} as
    \begin{equation*}
        \text{MAE} := \frac{1}{N*m*n}\sum\limits_{t=1}^N \sum\limits_{i,j=1}^{m,n} |(\mathbf{X}_t)_{ij} - (\hat{\mathbf{Y}}_t)_{ij} |. 
    \end{equation*}
    The \textbf{Mean Squared Error} is defined as 
    \begin{equation*}
        \text{MSE} = \frac{1}{N*m*n}\sum\limits_{t=1}^N \sum\limits_{i,j=1}^{m,n} ((\mathbf{X}_t)_{ij} - (\hat{\mathbf{Y}}_t)_{ij} )^2.
    \end{equation*}
    The \textbf{root mean square error (RMSE)} value is defined as 
    \begin{equation*}
        \text{RMSE} := \sqrt{\text{MSE}}.
    \end{equation*}
    The \textbf{symmetric mean absolute percentage error (SMAPE)} is defined as 
    \begin{equation*}
        \text{SMAPE} := \frac{100}{N * m * n} \sum\limits_{t = 1}^N \sum\limits_{i,j=1}^{m,n} \frac{|(\mathbf{X}_t)_{ij} - (\hat{\mathbf{Y}}_t)_{ij}|}{(|(\mathbf{X}_t)_{ij}| + |(\hat{\mathbf{Y}}_t)_{ij}|)\frac{1}{2} }.
    \end{equation*}
\end{definition}

\clearpage

\section{Appendix - proofs of theorems}

\renewcommand{\thetheorem}{B\arabic{theorem}}
\renewcommand{\thedefinition}{B\arabic{definition}}

\setcounter{theorem}{0}

\begin{theorem}
\label{twr:cov_matrix_permutation}
    Consider the matrices $\mathbf{A}\in \RR^{m\times m}$ and $\mathbf{B}\in \RR^{n\times n}$. The Kronecker product matrix $\mathbf{A}\otimes \mathbf{B} = [ab^\otimes_{ij}]_{i,j = 0, \ldots, mn-1}$ can be expressed in terms of the elements of the matrix $\vect(\mathbf{A})\vect(\mathbf{B})^T = [ab^v_{ij}]_{i,j = 0, \ldots, mn-1}$. Each element of the Kronecker product matrix is given by
    $$ab^\otimes_{ij} = ab^v_{(i\%m) * n + (j\%n), (i//m) * m + j//n}.$$
\end{theorem}
\begin{proof}
    Let $i$ and $j$ be an row and column index of matrix $\vect(\mathbf{A})\vect(\mathbf{B})^T$, and 
    $i^\otimes$ and $j^\otimes$ be an row and column index of matrix $\mathbf{A} \otimes \mathbf{B}$. Note that we numbering the matrix elements starting from zero. From the definition of the Kronecker product, i.e.,
    \begin{equation*}
        \mathbf{A}\otimes \mathbf{B} = \begin{bmatrix}
            a_{11} \mathbf{B} & \ldots & a_{1m} \mathbf{B} \\
            \vdots & \ddots & \vdots \\
            a_{m1}\mathbf{B} & \ldots & a_{mm}\mathbf{B} 
        \end{bmatrix}
    \end{equation*}
    we observe that
    $$(\mathbf{A}\otimes \mathbf{B})_{i^\otimes j^\otimes} = a_{i^\otimes //n, j^\otimes //m} b_{i^\otimes \%n, j^\otimes \%m},$$
    where $\%$ denotes the remainder, and $//$ truncating integer division \cite{kron_wolfram}.
    On the other hand, considering the product of two vectors, we obtain
    \begin{equation*}
        (\vect(\mathbf{A}) \vect(\mathbf{B})^T)_{ij} = a_{i\%m,i//m} b_{j\%n, j//n}.
    \end{equation*}
By comparing the indices of both matrices $(\mathbf{A}\otimes \mathbf{B})$ and $(\vect(\mathbf{A}) \vect(\mathbf{B})^T)$, we derive the following equations
\begin{align*}
    \begin{cases}
    i^\otimes//n &= i \% m \\
    j^\otimes //m &= i//m \\
    i^\otimes \% n &= j\% n \\
    j^\otimes \% m &= j //n. 
    \end{cases}
\end{align*}
We can solve the equation using the basic property of the modulus operator, i.e.,
$$i/n = i//n + i\%n,$$ 
by substituting the appropriate equivalences derived above.
Solving these equations yields
\begin{equation*}
    i^\otimes = (i\%m) * n + (j\%n), \quad j^\otimes = (i//m) * m + j//n. 
\end{equation*}
\end{proof}

\begin{theorem}
\label{twr:min_Ek}
Let $\mathbf{A} \in \RR^{m \times m}$, $\mathbf{B} \in \RR^{n \times n}$, and let $\mathbf{f}_t, \mathbf{b}_t \in \RR^{m \times n}$ for $t = 1, \ldots, N$. Assume that $||\mathbf{A}||_F = 1$. 
The solution to the minimization problem
    \begin{equation}
    \label{eq:burg_energy_func}
        \min_{\mathbf{A},\mathbf{B}} \sum_{t=k+1}^{N}\left( ||\mathbf{f}_t - \mathbf{A} \mathbf{b}_{t-1} \mathbf{B}^T||^2_F + ||\mathbf{b}_{t-1} - \mathbf{A} \mathbf{f}_t \mathbf{B}^T||^2_F \right)
    \end{equation}
is given by the following equations:
    \begin{align*}
            \mathbf{A} &=  
    \sum_{t=k+1}^{N} \left[
     \mathbf{f}_t\mathbf{B}\mathbf{b}_{t-1}^T +   \mathbf{b}_{t-1}\mathbf{B} \mathbf{f}_t^T\right]
     \left( \sum_{t=k+1}^{N} \left[ \left(\mathbf{B} \mathbf{b}_{t-1}^T\right)^2 +  \left(\mathbf{B}\mathbf{f}_t^T\right)^2 \right] \right)^{-1} \\
     \mathbf{B}^T &= \left( \sum_{t=k+1}^{N} \left[ \left(\mathbf{A}\mathbf{b}_{t-1}\right)^2 -  \left(\mathbf{A} \mathbf{f}_t\right)^2 \right] \right)^{-1}
    \sum_{t=k+1}^{N} \left[
     \mathbf{b}_{t-1}^T \mathbf{A}^T \mathbf{f}_t + \mathbf{f}_t^T \mathbf{A}^T \mathbf{b}_{t-1} \right].
    \end{align*}
\end{theorem}

\begin{proof}
By utilizing the properties of the Frobenius norm, specifically 
$$||\mathbf{A}+\mathbf{B}||^2_F = ||\mathbf{A}||^2_F + ||\mathbf{B}||^2_F + 2 \langle \mathbf{A}, \mathbf{B} \rangle_F$$ 
and 
$$\langle \mathbf{A}, \mathbf{B} \rangle_F = \sum\limits_{i,j} a_{ij}b_{ij} = \vect(\mathbf{A})^T \vect(\mathbf{B}),$$ 
we have 
\begin{align*}
    ||\mathbf{f}_t - \mathbf{A} \mathbf{b}_{t-1} \mathbf{B}^T||^2_F
    = ||\mathbf{f}_t||_F^2 + ||\mathbf{A} \mathbf{b}_{t-1} \mathbf{B}^T||_F^2 - 2 \vect\left(\mathbf{f}_t\right)^T \vect\left(\mathbf{A} \mathbf{b}_{t-1} \mathbf{B}^T\right)
\end{align*}
and  
\begin{align*}
    ||\mathbf{b}_{t-1} - \mathbf{A} \mathbf{f}_t \mathbf{B}^T||^2_F = ||\mathbf{b}_{t-1}||_F^2 + ||\mathbf{A} \mathbf{f}_t \mathbf{B}^T||_F^2 - 2 \vect\left(\mathbf{b}_{t-1}\right)^T \vect\left( \mathbf{A} \mathbf{f}_t \mathbf{B}^T \right).
\end{align*}
To express these equations in a more manageable form, we introduce the following notations
\begin{align*}
    f^I &= ||\mathbf{f}_t||_F^2, \\
    f^{II} &= ||\mathbf{A} \mathbf{b}_{t-1} \mathbf{B}^T||_F^2, \\
    f^{III} &= - 2 \vect\left(\mathbf{f}_t\right)^T \vect\left(\mathbf{A} \mathbf{b}_{t-1} \mathbf{B}^T\right), \\
    b^{I} &= ||\mathbf{b}_{t-1}||_F^2, \\
    b^{II} &= ||\mathbf{A} \mathbf{f}_t \mathbf{B}^T||_F^2, \\
    b^{III} &= - 2 \vect\left(\mathbf{b}_{t-1}\right)^T \vect\left( \mathbf{A} \mathbf{f}_t \mathbf{B}^T \right).
\end{align*}
To find the minimum of the function
$$E(\mathbf{A}, \mathbf{B}) := \sum\limits_{t=k+1}^{N} \left(||\mathbf{f}_t - \mathbf{A} \mathbf{b}_{t-1} \mathbf{B}^T||^2_F + ||\mathbf{b}_{t-1} - \mathbf{A} \mathbf{f}_t \mathbf{B}^T||^2_F \right),$$
we compute the derivatives with respect to $\mathbf{A}$ and $\mathbf{B}$ (see \cite{matrix_derivatives}). To simplify the calculations, we express the derivatives in terms of the vectorized forms of $\mathbf{A}$ and $\mathbf{B}^T$, i.e., $\vect(\mathbf{A})$ and $\vect(\mathbf{B}^T)$. We obtain
\begin{align*}
    \frac{\partial f^{I}}{\partial \vect(\mathbf{A})} &= 0^{m^2 \times 1}, \\
    \frac{\partial f^{I}}{\partial \vect(\mathbf{B}^T)} &= 0^{n^2 \times 1}.
\end{align*}
Using basic properties of the vectorization operation, we derive
\begin{align*}
    f^{II} &= \vect\left( \mathbf{A} \mathbf{b}_{t-1} \mathbf{B}^T\right)^T \vect\left( \mathbf{A} \mathbf{b}_{t-1} \mathbf{B}^T\right) = \\
    &= \vect(\mathbf{A})^T \left(\mathbf{B} \mathbf{b}_{t-1}^T \otimes I_m \right)^T \left(\mathbf{B} \mathbf{b}_{t-1}^T \otimes I_m \right) \vect(\mathbf{A}) = \\
    &= \vect(\mathbf{B}^T)^T \left( I_n \otimes \mathbf{A} \mathbf{b}_{t-1} \right)^T \left(I_n \otimes \mathbf{A} \mathbf{b}_{t-1} \right) \vect(\mathbf{B}^T).
\end{align*}
Hence, we obtain the derivative
\begin{align*}
    \frac{\partial f^{II}}{\partial \vect(\mathbf{A})} &= 2 \left(\mathbf{B} \mathbf{b}_{t-1}^T \otimes I_m \right)^T \left(\mathbf{B} \mathbf{b}_{t-1}^T \otimes I_m \right) \vect(\mathbf{A}) = 2\vect(\mathbf{A} (\mathbf{B}\mathbf{b}_{t-1}^T)^2), \\
    \frac{\partial f^{II}}{\partial \vect(\mathbf{B}^T)} &= 2 \left( I_n \otimes \mathbf{A} \mathbf{b}_{t-1} \right)^T \left(I_n \otimes \mathbf{A} \mathbf{b}_{t-1} \right) \vect(\mathbf{B}^T) = 2 \vect((\mathbf{A}\mathbf{b}_{t-1})^2\mathbf{B}^T).
\end{align*}
Next, applying a similar approach to $f^{III}$, we obtain
\begin{align*}
    f^{III} &= -2 \vect(\mathbf{f}_t)^T (\mathbf{B}\mathbf{b}_{t-1}^T \otimes I_m) \vect(\mathbf{A}) = \\
    &= -2 \vect(\mathbf{f}_t)^T (I_n \otimes \mathbf{A}\mathbf{b}_{t-1}) \vect(\mathbf{B}^T),
\end{align*}
and 
\begin{align*}
    \frac{\partial f^{III}}{\partial \vect(\mathbf{A})} &= -2 \left[ \vect(\mathbf{f}_t)^T (\mathbf{B}\mathbf{b}_{t-1}^T \otimes I_m) \right]^T = -2 \vect(\mathbf{f}_t \mathbf{B} \mathbf{b}_{t-1}^T), \\
    \frac{\partial f^{III}}{\partial \vect(\mathbf{B}^T)} &= -2 \left[\vect(\mathbf{f}_t)^T (I_n \otimes \mathbf{A}\mathbf{b}_{t-1})\right]^T = -2 \vect(\mathbf{b}_{t-1}^T\mathbf{A}^T\mathbf{f}_t).
\end{align*}
A similar calculation can be carried out for $b^I$, $b^{II}$, and $b^{III}$. We obtain
\begin{align*}
    \frac{\partial b^{I}}{\partial \vect(\mathbf{A})} &= 0^{m^2 \times 1} \\
    \frac{\partial b^{I}}{\partial \vect(\mathbf{B}^T)} &= 0^{n^2 \times 1} \\
    \frac{\partial b^{II}}{\partial \vect(\mathbf{A})} &= 2 \vect\left(\mathbf{A} \left(\mathbf{B} \mathbf{f}_t^T\right)^2\right) \\ 
    \frac{\partial b^{II}}{\partial \vect(\mathbf{B}^T)} &= 2 \vect\left(\left(\mathbf{A} \mathbf{f}_t\right)^2 \mathbf{B}^T\right) \\ 
    \frac{\partial b^{III}}{\partial \vect(\mathbf{A})} &= - 2 \vect\left(\mathbf{b}_{t-1}\mathbf{B}\mathbf{f}_t^T\right) \\
    \frac{\partial b^{III}}{\partial \vect(\mathbf{B}^T)} &= 2 \vect\left(\mathbf{f}_t^T \mathbf{A}^T \mathbf{b}_{t-1}\right).
\end{align*}
By computing the derivatives for all $t = k+1, \ldots, N$, we obtain the following
\begin{equation}
    \frac{\partial E(\mathbf{A},\mathbf{B})}{\partial \vect(\mathbf{A})} = \sum_{t=k+1}^{N} \left[ 2\vect\left(\mathbf{A} \left(\mathbf{B}\mathbf{b}_{t-1}^T\right)^2\right) + 2\vect\left(\mathbf{A} \left(\mathbf{B} \mathbf{f}_t^T\right)^2\right) - 2\vect\left(\mathbf{f}_t \mathbf{B} \mathbf{b}_{t-1}^T\right) - 2\vect\left(\mathbf{b}_{t-1}\mathbf{B}\mathbf{f}_t^T\right) \right], 
\end{equation}
\begin{equation}
    \frac{\partial E(\mathbf{A},\mathbf{B})}{\partial \vect(\mathbf{B}^T)} = \sum_{t=k+1}^{N} \left[ 2 \vect\left(\left(\mathbf{A}\mathbf{b}_{t-1}\right)^2 \mathbf{B}^T\right) + 2 \vect\left(\left(\mathbf{A} \mathbf{f}_t\right)^2 \mathbf{B}^T\right) - 2 \vect\left(\mathbf{b}_{t-1}^T \mathbf{A}^T \mathbf{f}_t\right) - 2 \vect\left(\mathbf{f}_t^T \mathbf{A}^T \mathbf{b}_{t-1}\right) \right].
\end{equation}
Next, by setting the derivative equal to zero, we have
\begin{align*}
    \sum_{t=k+1}^{N} \left[ \vect\left(\mathbf{A} \left(\mathbf{B}\mathbf{b}_{t-1}^T\right)^2\right) + \vect\left(\mathbf{A} \left(\mathbf{B} \mathbf{f}_t^T\right)^2\right) - \vect\left(\mathbf{f}_t \mathbf{B} \mathbf{b}_{t-1}^T\right) - \vect\left(\mathbf{b}_{t-1}\mathbf{B}\mathbf{f}_t^T\right) \right] &= 0^{m^2 \times 1}, \\
    \sum_{t=k+1}^{N} \left[ \mathbf{A} \left(\mathbf{B}\mathbf{b}_{t-1}^T\right)^2 + \mathbf{A} \left(\mathbf{B} \mathbf{f}_t^T\right)^2 - \mathbf{f}_t \mathbf{B} \mathbf{b}_{t-1}^T - \mathbf{b}_{t-1}\mathbf{B}\mathbf{f}_t^T \right] &= 0^{m \times m}.
\end{align*}
Solving the resulting equation with respect to $\mathbf{A}$, we obtain the following
\begin{align*}
    \mathbf{A} \sum_{t=k+1}^{N} \left[ \left(\mathbf{B}\mathbf{b}_{t-1}^T\right)^2 + \left(\mathbf{B} \mathbf{f}_t^T\right)^2 \right] &= \sum_{t=k+1}^{N} \left[ \mathbf{f}_t \mathbf{B} \mathbf{b}_{t-1}^T + \mathbf{b}_{t-1}\mathbf{B}\mathbf{f}_t^T \right], \\ 
    \mathbf{A} &= \sum_{t=k+1}^{N} \left[ \mathbf{f}_t \mathbf{B} \mathbf{b}_{t-1}^T + \mathbf{b}_{t-1}\mathbf{B}\mathbf{f}_t^T \right] \left( \sum_{t=k+1}^{N} \left[ \left(\mathbf{B}\mathbf{b}_{t-1}^T\right)^2 + \left(\mathbf{B} \mathbf{f}_t^T\right)^2 \right] \right)^{-1}.
\end{align*}
Similarly, we obtain that
\begin{equation*}
    \mathbf{B}^T = \left( \sum_{t=k+1}^{N} \left[ \left(\mathbf{A}\mathbf{b}_{t-1}\right)^2  + \left(\mathbf{A} \mathbf{f}_t\right)^2 \right] \right)^{-1} \sum_{t=k+1}^{N}\left[ \mathbf{b}_{t-1}^T \mathbf{A}^T \mathbf{f}_t + \mathbf{f}_t^T \mathbf{A}^T \mathbf{b}_{t-1} \right].
\end{equation*}
We have found a stationary point of the function $E(\mathbf{A}, \mathbf{B})$. To show that this point corresponds to a global minimum, it is sufficient to demonstrate that the function $E$ is convex. Indeed, we have 
$$||\mathbf{f}_t - \mathbf{A} \mathbf{b}_{t-1} \mathbf{B}^T||^2_F = ||\vect(\mathbf{f}_t) - (\mathbf{B} \otimes \mathbf{A}) \vect(\mathbf{b}_{t-1})||^2_F.$$ 
Thus, each term in the sum in equation (\ref{eq:burg_energy_func}) is a convex quadratic function (see \cite{conv_opt}). Since the sum of convex functions is also convex, the function $E$ is convex, and the stationary point is therefore a global minimum.
\end{proof}

\bibliographystyle{plainnat}
\bibliography{bibliography}

\end{document}